\documentclass[12pt, english]{amsart}

\usepackage{general}

\title[A conjecture by Kollár and Saccà]{On a conjecture by Kollár and Saccà}

\author[P. Autissier]{Pascal Autissier}
\address{Univ. Bordeaux, CNRS, Bordeaux INP, IMB, UMR 5251, F-33400 Talence, France}
\curraddr{}
\email{pascal.autissier@math.u-bordeaux.fr}

\author[A. Fanelli]{Andrea Fanelli}
\address{Univ. Bordeaux, CNRS, Bordeaux INP, IMB, UMR 5251, F-33400 Talence, France}
\curraddr{}
\email{andrea.fanelli@math.u-bordeaux.fr}

\subjclass[2020]{}
\thanks{}
\date{\today}

\begin{document}


\begin{abstract}
	In this note we study smooth commutative group schemes over curves, whose generic fibre is an abelian variety. We prove a modified version of the conjecture proposed in \cite{KS25}.
\end{abstract}

\maketitle
\tableofcontents


\section*{Introduction}

The goal of this note is prove the following statement for smooth commutative group schemes over smooth curves, conjectured by Kollár and Saccà in \cite{KS25}.

\begin{theorem*}(=\autoref{cor_main})
Let $C$ be a geometrically connected, smooth projective curve over a field $k$, and $A/C$
a smooth commutative group scheme with connected fibres, whose generic fibre $A_\eta$ is an abelian
variety which has no abelian subvarieties defined over $k$. 
\\Let $Z_0$ be the zero section and $MW(A/C)$ the Mordell-Weil group. Then the following set map
$$\MW(A/C)/{\tors} \to  N_{1}(A)$$
$$Z \mapsto [Z]-[Z_0]$$
is injective.
\end{theorem*}

We remark that the previous natural map is \emph{not} a group homomorphism and to prove the injectivity-modulo-torsion we introduce a weakened version of numerical equivalence on 1-cycles (see \autoref{def_N1gen} and \autoref{main_thm}).

The key inputs of our approach are
\begin{itemize}
\item extension results for cubist $\mathbb{G}_{m}$-torsors on connected group schemes (see \autoref{sec_cubist}); and 
\item the theory of heights developed for the proof of the Lang-Néron theorem (see \cite[Chapter~6]{L83}).
\end{itemize}

\subsection*{Acknowledgements} The  authors  would  like  to  thank  János Kollár and Giulia Saccà for their comments on the first version of this note. AF is supported by the ANR project ``FRACASSO''  ANR- 22-CE40-0009-01.

\section{Extensions of cubist invertible sheaves} \label{sec_cubist}
Let $k$ be a field and $S$ an irreducible normal scheme over $k$. In this note we denote by $A$ a smooth commutative group scheme over $S$. 
If we assume that the generic fibre $A_\eta$ is an abelian variety, all line bundles on $A_\eta$ verify the Theorem of the Cube \cite[Corollary~2, page 58]{M70}. In particular, it is natural to ask under which conditions on $A$ (and $S$), a line bundle $\L_\eta$ on $A_\eta$ can be extended to $\L$ on $A$ still verifying the Theorem of the Cube. It turns out, after the work of Breen \cite{B83} (see also \cite{M85}), that when $S$ is a curve and $A/S$ has connected fibres, a (unique) \emph{cubist} extension exists for any line bundle on $A_\eta$.  

The notion of $G$-torsor endowed with a cubist structure (or cubist $G$-torsor), where $G$ is a commutative group scheme, is rather general and we will introduce here only what is relevant for our work (see \cite[Ch.~I, Definition~2.4.5]{M85}).

\begin{definition}
Let $S$ be a scheme, $\mathbb{G}_{m,S}$ the multiplicative group scheme over $S$ and $X$ a commutative $S$-group scheme. A \emph{cubist structure} on an invertible sheaf $\L$ of $X$ (or $\mathbb{G}_{m,S}$-torsor on $X/S$) is the data of a section $\tau$ of the torsor
$$\theta(\L):= m_{1,2,3}^*\L \otimes m_{1,2}^*\L^\vee \otimes m_{1,3}^*\L^\vee \otimes m_{2,3}^*\L^\vee \otimes m_{1}^*\L \otimes m_{2}^*\L \otimes m_{3}^*\L$$
on $X^3$, where $m_I\colon X^3 \to X$ is the sum of projections corresponding to $I\subseteq \{1,2,3\}$.
The category of cubist $\mathbb{G}_{m,S}$-torsors on $X$ is denoted by $\CUB(X,\mathbb{G}_{m,S})$.
\end{definition}

We recall the classical notion of rigidification for line bundles.

\begin{definition}
Let $k$ be a field, $S$ an irreducible normal scheme over $k$ and $A/S$ a smooth commutative group scheme, with connected fibres. Let $Z_0$ be the zero section. An invertible sheaf $\L$ on $A$ is \emph{rigidified} if $\L|_{Z_0}\cong \O_{Z_0}$. The group of rigidified invertible sheaves on $A$ modulo isomorphism is denoted by $\Pic(A)_{\rig}$.
\end{definition}

\begin{proposition}\label{prop_main}
Let $C$ be a smooth projective curve over $k$ and $A/C$ a smooth commutative group scheme with connected fibres. Assume that the generic fibre $A_\eta$ is an abelian variety. Then
\begin{enumerate}
    \item\label{item_lift} the restriction $$\Pic(A)_{\rig} \to \Pic(A_\eta)$$ is a group isomorphism.
\item\label{item_sum} For any $\L\in \Pic(A)_{\rig}$ the following holds:
$$\L_\eta\in \Pic^\circ(A_\eta) \iff [-1]^*\L \cong \L^\vee \text{(i.e.\ $\L$ is odd)}.$$
\item\label{item_square} Let $\L\in \Pic(A)$ be a line bundle such that $\L_\eta\in \Pic^\circ(A_\eta)$, and let $Z_1$ and $Z_2$ be two sections of $A/C$, and $Z_3:=Z_1\oplus Z_2$ in $\MW(A/C)$ with corresponding translation morphisms $\tau_i\colon A\to A$, for $i=1,2,3$.
Then
$$\tau_1^*\L \otimes \tau_2^*\L \cong \tau_3^*\L \otimes \L$$
in $\Pic(A)$.
\end{enumerate}
\end{proposition}

\begin{proof}
Our hypothesis on the base $C$ guarantee the existence and unicity of cubist extensions, i.e.\ the restriction functor 
\begin{equation}\label{eq_cub}
    \CUB(A,\mathbb{G}_{m,C}) \to \CUB(A_\eta,\mathbb{G}_{m,\eta})
\end{equation}
is an equivalence of categories (see \cite[Ch.~2, Theorem 1.1]{M85}). Composing with the forgetful functor $$\CUB(A,\mathbb{G}_{m,C}) \to \TORSRIG(A,\mathbb{G}_{m,C}),$$ where $\TORSRIG(A,\mathbb{G}_{m,C})$ is the category of rigidified (i.e.\ trivialised at the zero section) $\mathbb{G}_{m,C}$-torsors, we obtain \autoref{item_lift}.
\\The theory of abelian varieties gives: $\L_\eta\in \Pic^\circ(A_\eta)$ if and only if $[-1]^*\L_\eta \cong \L_\eta^\vee$ (see \cite[Ch.~5, Proposition~2.3]{L83}). So \autoref{item_lift} implies \autoref{item_sum}.
\\Let $\L\in \Pic(A)$. Let $p_1$ and $p_2$ the two projections from $A\times A$. Restricting to the generic fibre, $\L_\eta\in \Pic^\circ(A_\eta)$ implies that $(p_1+p_2)^*\L_\eta \cong p_1^*\L_\eta \otimes p_2^*\L_\eta$ on $A_\eta\times A_\eta$, so, again by \autoref{eq_cub}, $p_1^*\L \otimes p_2^*\L \cong (p_1+p_2)^*\L \otimes 0^*\L.$
Pulling-back the previous equation via $(f,g)$, where $f,g\colon A \to A$ are morphisms, we get 
$$f^*\L \otimes g^*\L \cong (f+g)^*\L \otimes 0^*\L.$$
This implies \autoref{item_square}.
\end{proof}

\section{The conjecture} \label{sec_conjecture}

In \cite{KS25}, the authors prove the following rigidity result, motivated by the work \cite{BFK25}.

\begin{theorem}\cite[Proposition~1]{KS25}\label{thm_KS}
Let $S$ be a smooth, projective surface over $\mathbb{C}$ such that $\Pic(S) = \Z[H]$, where $|H|$ is basepoint-free, and members of $|H|$ have at worst nodes in codimension
1 on $|H|$. Let $p \colon J(S, H) \to |H|$ be the universal compactified Jacobian,
$L \subset |H|$ a general line, $J_L := p^{-1}
(L)$, and $g$ the genus of the curves in $|H|$. Let $Z \subset J_L$ be a section whose cohomology class is contained in the image of
the restriction map 
$$H^{2g}(J(S, H), \Z) \to H^{2g}(J_L, \Z).$$
Then $Z$ is the zero section.
\end{theorem}

The proof reduces to a monodromy argument combined with an injectivity statement for the map
\begin{equation}\label{eq_1}
\MW(J_L/L)\cong \Z^{r-1} \to N_1(J_L)
\end{equation}
$$Z\mapsto [Z]-[Z_0]$$
where $N_1$ denotes the group of complete $1$-cycles modulo numerical equivalence. 

One is induced to consider the previous map in a generalised setting (see \cite[Conjecture~7]{KS25}). First, we remark that the map \autoref{eq_1} is \emph{not} a group homomorphism.

\begin{example}
Keep the same hypothesis as in \autoref{thm_KS}. Choose any non-zero $Z\in \MW(J_L/L)$ (the $K/\mathbb{C}$-trace of $J_L$ is automatically trivial and $\MW(J_L/L)$ is torsion-free, see \cite[Theorem~3]{S99}) and any $\L\in \Pic(J_L)$ relatively ample over $L$, rigidified and even (i.e.\ such that $[-1]^*\L\cong\L$). 
We know that $[2]^*\L=\L^{\otimes 4}$. Moreover, we remark that $(\L\cdot Z)= h_\L(Z_\eta)$, where $h_\L$ is the canonical Néron-Tate height (see \cite[Ch.~12, Proposition~3.5]{L83}), so $(\L\cdot Z)>0$. By the projection formula, $([2]^*\L\cdot Z)=(\L\cdot [2]_* Z)$ and we deduce that
$$(\L\cdot [2]_* Z)=4(\L\cdot Z).$$
Since $(\L\cdot 2Z)=2(\L\cdot Z)<4(\L\cdot Z)$, we have proved that 
$$2[Z]\neq [[2]_*Z]+[Z_0] \ \text{ in } N_1(J_L).$$
In particular, the map \autoref{eq_1} is not a group homomorphism.
\end{example}

\noindent It turns out that the previous example explains the only obstruction for \autoref{eq_1} to be a homomorphism.

\subsection{The natural map}
Let $p \colon A \to C$ be a smooth commutative group scheme over a smooth projective curve $C$ with connected fibres. Consider the natural set map \autoref{eq_1} in this general setting:
\begin{equation}\label{eq_2}
\phi\colon \MW(A/C) \to N_1(A)
\end{equation}
$$Z\mapsto [Z]-[Z_0]$$
This map is quadratic in the following sense: if we define
$$b\colon \MW(A/C)\times \MW(A/C) \to N_1(A)$$
$$(W,Z) \mapsto [W\oplus Z] + [Z_0] -[W] -[Z]$$
the Theorem of the Cube implies that $b$ is bilinear. Moreover the map
$$\ell\colon \MW(A/C)\to N_1(A)$$
$$Z \mapsto 4[Z] - [Z \oplus Z] - 3[Z_0]$$
is linear, applying \autoref{prop_main}(3) to the line bundle $\L^{\otimes 4}\otimes [2]^*\L^\vee$. By construction,
$$2\phi(Z)=b(Z,Z)+\ell(Z),$$
for all $Z\in\MW(A/C)$.

\noindent In order to obtain a linear map from \autoref{eq_2}, we define a weakened version of numerical equivalence.

\begin{definition}\label{def_N1gen}
Let $p \colon A \to C$ be a smooth commutative group scheme with connected fibres. The \emph{generic numerical equivalence} on (complete) $1$-cycles is defined as
$$Z\equiv_{\gen}0 \text{ if } \ (\L\cdot Z)=0 \ \text{for all } \L \in \Pic(A) \text{ such that\ } \L_{\eta} \in \Pic^\circ(A_{\eta}).$$
The group of complete $1$-cycles modulo generic numerical equivalence is denoted by $N_{1,\gen}(A)$.
\end{definition}

\noindent Our main result is the following.

\begin{theorem}\label{main_thm}
Let $C$ be a geometrically connected, smooth projective curve over a field $k$, and $A/C$
a smooth commutative group scheme with connected fibres, whose generic fibre $A_\eta$ is an abelian
variety. Let $Z_0$ be the zero section. Then the following holds.
\begin{enumerate}
\item The map \begin{equation}
\psi\colon \MW(A/C) \to  N_{1,\gen}(A)
\end{equation}
$$Z \mapsto [Z]-[Z_0]$$
is a group homomorphism.
\item Assume that $A_\eta$ has no abelian subvarieties defined over $k$, then $\ker \psi = \MW(A/C)_{\tors}$.
\end{enumerate}
\end{theorem}

\begin{proof}
To prove that $\psi$ is a group homomorphism, we follow the first part of the argument from \cite[§8]{KS25}. Let $Z_3=Z_1\oplus Z_2\in \MW(A/C)$, we want to show that $\psi(Z_1)+\psi(Z_2) \equiv_{\gen}\psi(Z_3)$, i.e.
\begin{equation}\label{eq_num}
    (\L\cdot Z_1)+(\L\cdot Z_2)=(\L\cdot Z_3)+(\L\cdot Z_0),
\end{equation}
for any $\L \in \Pic(A)$ such that $\L_{\eta} \in \Pic^\circ(A_{\eta})$. 
\\Let us denote by $\tau_i\colon A\to A$ the translation by $Z_i$, then the projection formula implies that \autoref{eq_num} is equivalent to 
$$(\tau_1^*\L \otimes \tau_2^*\L\otimes \tau_3^*\L^\vee\otimes \L^\vee\cdot Z_0)=0,$$
which holds true by \autoref{prop_main}\autoref{item_square}.
\\To describe the kernel of $\psi$, we recall once again that $(\L\cdot Z)= h_\L(Z_\eta)$, for all $\L\in\Pic(A)_{\rig}$, where $h_\L$ is the canonical Néron-Tate height (see \cite[Ch.~12, Proposition~3.5]{L83}). Let $Z\in \MW(A/C)$ verifying $(\L\cdot Z)=0$ for all $\L\in\Pic(A)_{\rig}$ such that $\L_\eta\in\Pic^\circ(A_\eta)$; then \cite[Ch.~6, Theorem~5.4.2)]{L83} implies that $Z\in \MW(A/C)_{\tors}$, since the $K/k$-trace vanishes by the hypothesis (we assumed that $A_\eta$ has no abelian subvarieties defined over $k$).
\end{proof}

\begin{corollary}\label{cor_main}
Let $A/C$ be as in \autoref{main_thm} and assume that $A_\eta$ has no abelian subvarieties defined over $k$. Then the set map
$$\MW(A/C)/{\tors} \to  N_{1}(A)$$
$$Z \mapsto [Z]-[Z_0]$$
is injective.
\end{corollary}

\begin{proof}
\autoref{main_thm} implies that, for any $Z_1,Z_2\in \MW(A/C)$,
$$[Z_1]=[Z_2] \text{ in } N_1(A) \iff Z_1\ominus Z_2\in \MW(A/C)_{\tors}.$$
The statement follows.
\end{proof}

\begin{remark}
The previous corollary is precisely the injectivity-modulo-torsion statement from \cite[Conjecture~7]{KS25}.
\end{remark}



\bibliographystyle{amsalpha}
\bibliography{biblio}

\end{document}